\documentclass{amsart}

\usepackage{amssymb,color}
\usepackage{amsfonts}
\usepackage{amsmath}
\usepackage{euscript}
\usepackage{enumerate}
\usepackage{graphics}

\newtheorem{theorem}{Theorem}[section]
\newtheorem{lemma}[theorem]{Lemma}
\newtheorem{note}[theorem]{Note}

\newtheorem{cor}[theorem]{Corollary}

\newtheorem*{Theorem1'}{Theorem 1'}

\theoremstyle{definition}

\theoremstyle{remark}

\numberwithin{equation}{section}

\newcommand \Q{{\mathbb Q}}
\newcommand \Ci{{\mathbb C}}

\newcommand \chr{{\mathrm {char}}}

\newcommand \C{{\mathcal C}}
\newcommand \R{{\mathbb R}}

\newcommand \B{{\mathcal B}}
\newcommand \GL{{\mathrm {GL}}}

\newcommand \si{{\sigma}}

\begin{document}

\title [Equivalence and normal forms of bilinear forms] {Equivalence and normal forms of bilinear forms}



\author{Fernando Szechtman}
\address{Department of Mathematics and Statistics, Univeristy of Regina, Canada}
\email{fernando.szechtman@gmail.com}
\thanks{The author was supported in part by an NSERC discovery grant}






\begin{abstract} We present an alternative account of the
problem of classifying and finding normal forms for arbitrary
bilinear forms. Beginning from basic results developed by Riehm,
our solution to this problem hinges on the classification of
indecomposable forms and in how uniquely they fit together to
produce all other forms. We emphasize the use of split forms,
i.e., those bilinear forms such that the minimal polynomial of the
asymmetry of their non-degenerate part splits over ground field,
rather than restricting the field to be algebraically closed. In
order to obtain the most explicit results, without resorting to
the classification of hermitian, symmetric and quadratic forms, we
merely require that the underlying field be quadratically closed.
\end{abstract}

\maketitle

\section{Introduction}

The problem of classifying arbitrary bilinear forms up to
equivalence, that is, arbitrary square matrices up to congruence,
has been solved through the work of Williamson \cite{W1},
\cite{W2}, \cite{W3}, Wall \cite{W}, Riehm \cite{R} and Gabriel
\cite{G}.

Over a general field, the solution consists of reducing the
classification to the case of hermitian and symmetric forms,
although in characteristic 2 this reduction involves the
classification of quadratic forms, as well. It would perhaps be
fair to refer to this as a relative solution.

Over an algebraically closed field it is possible to solve the
equivalence problem explicitly, as done by Riehm \cite{R}.
Moreover, very simple normal forms have been obtained in this case
by Horn and Sergeichuk \cite{HS3}.

Somewhere in between lies the case of {\em split} forms. We define
these as bilinear forms $f$ such that the minimal polynomial of
the asymmetry of the non-degenerate part of $f$ splits over ground
field (these terms are defined below). Of course, every bilinear
form defined over an algebraically closed field is split, but this
requirement is perhaps too demanding. All bilinear forms
considered in this paper will be assumed to be split. This is
similar in spirit to Jacobson's \cite{J} decision to consider
split semisimple Lie algebras in characteristic 0, rather than
restricting the ground field to be algebraically closed.

In this paper, we furnish an alternative account of how the classification problem
can be solved and how normal forms
can be produced, in the case of split forms. The end results are explicit and the means to
arrive at them are extremely simple. To achieve these goals we
must sacrifice generality by requiring the ground field to be
quadratically closed at certain strategic points.

Details of our strategy and prior work on the subject are
discussed below. Let us start, however, by reviewing some of the
known results on the classification of bilinear forms in the
classical case of alternating and symmetric forms. In either case,
after splitting the radical, one is reduced to consider
non-degenerate forms only.

It is well-known that a finite dimensional vector space $V$ admits
a non-degenerate alternating form if and only if $\dim(V)$ is
even, in which case any two such forms defined on $V$ are
equivalent.

For non-degenerate symmetric bilinear forms the classification is
field dependent. Let $f,g:V\times V\to F$ be two such forms, where
$F$ is a field and $V$ is a finite dimensional vector space over
$F$. Let us write $f\sim g$ to mean that $f$ and $g$ are
equivalent.

Suppose first that $F$ is quadratically closed. If $\chr(F)\neq 2$
then $f\sim g$. If $\chr(F)=2$ then $f\sim g$ if and only if $f$
and $g$ are both alternating or both non-alternating.

Over the reals, $f\sim g$ if and only if they have the same
signature; over a finite field the decisive condition is the
discriminant in characteristic not 2, and simply the alternating
or non-alternating nature of the forms in characteristic 2; over
the $p$-adic field $\Q_p$ the invariants are the discriminant and
the Hasse symbol, while $f\sim g$ over $\Q$ if and only if $f\sim
g$ over $\R$ and $\Q_p$ for every prime $p$. See \cite{K} for
details about the above cases, except the last which can be found
in \cite{O}, which also treats the more general cases of arbitrary
local and global fields.

Let us move to the case of arbitrary bilinear forms. The starting
point is the following result of Gabriel \cite{G}: Any bilinear
form, say $f$, decomposes as the orthogonal direct sum of finitely
many indecomposable degenerate forms and a non-degenerate form.
Moreover, all summands in this decomposition are uniquely
determined by~$f$, up to equivalence. The uniqueness is really
only up to equivalence, as shown by Djokovic and Szechtman
\cite{DS2} in their study of the isometry group of an arbitrary
bilinear form. In this regard, see \cite{S} and \cite{D}.
Gabriel's result effectively reduces the classification problem to
the case of non-degenerate forms. In particular, Gabriel
determines all indecomposable degenerate forms; there is only one
for each dimension, and we will refer to it as a Gabriel block.
The above formulation of Gabriel's result was given by Waterhouse
\cite{WW}, who also included a matrix appearance of Gabriel's
blocks. Waterhouse used Gabriel's decomposition to compute the
number of congruence classes in $M_n(F_q)$, following Gow
\cite{RG}, who had previously obtained the corresponding result
for $\GL_n(F_q)$, perhaps unaware of \cite{G}. In his proof,
Gabriel uses the theory of Kronecker modules (better known to
linear algebraists as the theory of matrix pencils) and in
particular the Krull-Remak-Schmidt theorem for these modules. An
elementary proof of the uniqueness part of Gabriel's result was
given by Djokovic and Szechtman \cite{DS}, in the more general
context of sesquilinear forms over semisimple artinian rings with
an involution. A simple proof of the existence of Gabriel's
decomposition was furnished by Djokovic, Szechtman and Zhao
\cite{DSZ}. An alternative new short proof is given in
\S\ref{sec:existence}. We note that \cite{DSZ} gives an algorithm
that given any  $A\in M_n(F)$ produces $X\in\GL_n(F)$ such that
$X^2=I$ and $X'AX=A'$, the transpose of $A$. The existence of such
$X$ had first been obtained by Gow \cite{RG2} for $A\in\GL_n(F)$,
and then by Yip and Ballentine \cite{YB} for $A\in M_n(F)$. Again,
the missing link was Gabriel's decomposition. We remark that
\cite{YB} is a continuation of prior work of Yip and Ballentine
\cite{YB2} on the equivalence of bilinear forms. A short proof of
the existence of $X$ for any $A\in M_n(F)$ was given by Horn and
Sergeichuk \cite{HS} using the study made by Sergeichuk \cite{VS}
of the classification of bilinear forms. An algorithmic view of
Gabriel's decomposition is given by Horn and Sergeichuck in
\cite{HS2}.

Let us consider next the case of non-degenerate bilinear forms
and, more generally, the case of a non-degenerate sesquilinear
form $f:V\times V\to D$ defined over a finite dimensional vector
space $V$ over a division ring $D$ with involution~$J$. Then $f$
admits a unique asymmetry $\si\in\GL(V)$, satisfying
$$
f(v,u)^J=f(u,\si v),\quad u,v\in V,
$$
which exerts considerable influence on $f$. Wall \cite{W} studied
the conjugacy problem in classical groups over division rings,
which he reduced to the classification problem of non-degenerate
sesquilinear forms, which he further reduced to the classification
of non-degenerate hermitian forms, by means of $\si$, although his
results are complete only in characteristic not 2. Wall's work
extends to division rings prior and decisive work by Williamson
\cite{W1}, \cite{W2}, \cite{W3}, written exclusively in matrix
form, over fields of characteristic not 2. Riehm \cite{R}, making
an alternative use of~$\si$, considers the equivalence problem of
non-degenerate bilinear forms over fields, and succeeds in
obtaining a reduction to a classical problem in all cases,
resorting to symmetric and quadratic forms in characteristic 2. A
continuation of Riehm's work, in collaboration with
Shrader-Frechette \cite{RF}, deals with the classification of
sesquilinear forms over a semisimple artinian ring with an
anti-automorphism, not necessarily involutive. Another point of
view of the classification problem of bilinear forms, also using
$\si$, is offered by Scharlau \cite{RS}. An entirely different
approach, by means of quivers and their representations, was used
by Sergeichuk in \cite{VS}, \cite{VS2} in characteristic not 2,
and through direct matrix computations \cite{VS3} in
characteristic 2.

In addition to having necessary and sufficient conditions for two
bilinear forms to be equivalent, it is desirable to have a list of
representatives for matrices under congruence, as a direct sum of
carefully selected indecomposable matrices. For an algebraically
closed field of characteristic not 2, this was done by Corbas and
Williams \cite{CW}. Much simpler representatives where obtained by
Horn and Sergeichuk \cite{HS} over $\Ci$, by simplifying prior
work by Sergeichuck \cite{VS}. Horn and Sergeichuk later extended
their work and produced a list of matrix representatives under
congruence over an arbitrary algebraically closed field
\cite{HS3}. Their canonical forms have already found various
applications; see \cite{ACS}, \cite{S}, \cite{D}, \cite{DZ}, for
instance.

Let us now outline our own approach to solve the classification
problem of split forms and to produce split normal forms over a
quadratically closed field. Essentially, we first find all
possible classes of indecomposable split forms, including suitable
Gram matrices for them, and then determine in what sense an
arbitrary split form decomposes uniquely in terms of these
components. Uniqueness becomes an issue only in characteristic 2.

By means of Gabriel's decomposition we may reduce to the case when
the split form $f:V\times V\to F$ is non-degenerate. Let
$\si\in\GL(V)$ be the asymmetry of~$f$ and view $V$ as an
$F[X]$-module via $\si$. Following Riehm, it is rather easy to
restrict attention to the case when $\si$ has only one eigenvalue,
namely 1 or -1. Thus the minimal polynomial of $\si$ is $p^r$,
where $r\geq 1$ and $p=X\pm 1$. A second easy simplification by
Riehm allows us to write
$$
V=V_1\perp\cdots\perp V_r,
$$
where each $V_m$, $1\leq m\leq r$, is a free $F[X]/(p^m)$-module.
More importantly, perhaps, is the fact, also shown by Riehm, that
while these summands are not unique, the equivalence class of the
restriction of $f$ to each $V_m$ is determined by that of $f$.
Details can be found in \S\ref{sec:riehm}.

We may thus focus on the case $V=V_m$, when all elementary
divisors of $\si$ are equal to $p^m$ for a fixed $m\geq 1$ and
$p=X\pm 1$. This is done in \S\ref{sec:oneortwo}, \S\ref{sec:one}
and \S\ref{sec:two}.

In \S\ref{sec:oneortwo} we take a closer look at the
non-degenerate bilinear $\widehat{f}$ on $V/pV$, also considered
by Riehm, and defined by
$$
\widehat{f}(u+pV,v+pV)=f(p^{m-1}u,v),\quad u,v\in V.
$$
It is easily seen to have scalar asymmetry given by $(-1)^{m-1}$
if $p=X-1$ and $(-1)^{m}$ if $p=X+1$. We show in
\S\ref{sec:oneortwo} that the asymmetry of an indecomposable
component $g$ of $f$ has two elementary divisors if $\widehat{g}$
is alternating and one elementary divisor otherwise. Notice that
if $\chr(F)\neq 2$ then this dichotomy is entirely determined by
$p$ and the parity of $m$. Thus, if $\chr(F)\neq 2$, then
$\widehat{f}$ is alternating if and only if $p=X-1$ and $m$ is
even, or $p=X+1$ and $m$ is odd. If $\chr(F)=2$ and $m$ is even an
old argument of Wall \cite{W} ensures that $\widehat{f}$ is
alternating. However, if $\chr(F)=2$ and $m$ is odd, then a
non-alternating $\widehat{f}$ may have alternating and
non-alternating components. Nonetheless, since a non-alternating
symmetric bilinear form can always be diagonalized \cite{K}, we
can always choose these components to be non-alternating. Thus, as
long as we can prove, for $F$ quadratically closed, that every
indecomposable component of $f$ is determined by its own
asymmetry, this will automatically imply the result for $f$,
provided we also know, in characteristic~2, the alternating or
non-alternating nature of the symmetric form on $V_m/(X-1)V_m$
associated to $f$, for every odd $m$ for which $V_m\neq 0$. It is
then an easy matter to produce matrix representatives for every
indecomposable bilinear form, and hence for every form.
A different problem altogether is to find matrix representatives that have a particularly
pleasant shape. This is discussed in \S\ref{sec:one}.

The fact that an indecomposable non-degenerate bilinear form $f$
whose asymmetry $\si$ has a single eigenvalue $\pm 1$ is
determined by its asymmetry is proven in \S\ref{sec:one} in the
non-alternating case, and in \S\ref{sec:two} in the alternating
case. These results, specially for $\widehat{f}$ non-alternating,
are considerably harder to prove than the other results of
 this paper.
In the alternating case, if $\chr(F)\neq 2$, no restrictions
are required from~$F$, while if $\chr(F)=2$ it suffices to assume
that~$F$ is quadratically closed, which also suffices in the
non-alternating case. The actual technical restrictions imposed on
$F$ are discussed in these sections.

We combine the above information to produce the desired classification and normal forms in
\S\ref{sec:main}.

\section{Preliminaries}
\label{sectermnot}

We fix throughout an arbitrary field $F$, a non-zero vector space
$V$ of finite dimension~$n$ over~$F$, and a bilinear form
$f:V\times V\to F$.

Two matrices $A,B\in M_n(F)$ are said to be congruent if there is
$X\in\GL_n(F)$ such that
$$B=X' A X,$$
where $X'$ stands for the transpose of $X$. In this case, we write
$A\sim B$.

If $\B=\{v_1,\dots,v_n\}$ is a basis of $V$ the Gram matrix $A\in
M_n(F)$ of $f$ relative to $\B$ is defined by $A_{ij}=f(v_i,v_j)$.
Notice that $A$ can be characterized as the only matrix in
$M_n(F)$ satisfying:
$$
[u]' A [v]=f(u,v),\quad u,v\in V,
$$
where $[u],[v]$ are the column vectors in $F^n$ formed by the
coordinates of $u,v$ with respect to $\B$. If
$\C=\{u_1,\dots,u_n\}$ is also basis of $V$, let $B$ stand for
corresponding Gram matrix of $f$. Let $X$ be the change of basis
matrix whose $i$th column is formed by the coordinates of $u_i$
with respect to $\B$. Then  $B=X' A X$. Thus, the Gram matrices of
$f$ with respect to different bases are congruent.

A bilinear form $g:W\times W\to F$ is said to be equivalent to $f$
if there exists a linear isomorphism $x:V\to W$ such that
$$
g(x u,x v)=f(u,v),\quad u,v\in V.
$$
In this case, if $\B$ be a basis of $V$, then $\C=x\B$ is a basis
of $W$, and the Gram matrices of $f$ and $g$ with respect to these
bases are identical. Conversely, if $g$ and $f$ admit the same
Gram matrices relative to bases $\B$ and $\C$ then $f$ and $g$ are
equivalent via the linear isomorphism that sends $\B$ onto $\C$.
Thus $f$ and $g$ are equivalent if and only if their Gram matrices
with respect to any bases are congruent.

For a subspace $U$ of $V$ we define
$$
L(U)=\{v\in V\, f(v,U)=0\},R(U)=\{v\in V\, f(U,v)=0\}\text{ and
}U^\perp=L(U)\cap R(U).
$$
The left and right radicals of $f$ are defined to be $L(V)$ and
$R(V)$, respectively, while the radical of $f$ is
$\mathrm{Rad}(V)=L(V)\cap R(V)=V^\perp$.

We say that $f$ is non-degenerate if $L(V)=0$. This equivalent to
$R(V)=0$. Both mean that $f$ admits an invertible Gram matrix.




Given subspaces $U_1$ and $U_2$ of $V$ we write
$$
V=U_1\perp U_2
$$
to mean that $V=U_1\oplus U_2$ and $f(U_1,U_2)=0=f(U_2,U_1)$. If
$V$ admits no decomposition $V=U_1\perp U_2$ except when $U_1=0$
or $U_2=0$ we say that $f$ is indecomposable.

Likewise, $A\in M_n(F)$ is said to be indecomposable under
congruence if $A$ is not congruent to the direct sum
$$
A=B\oplus C=\left(%
\begin{array}{cc}
  B & 0 \\
  0 & C \\
\end{array}%
\right)
$$
of two smaller square matrices $B$ and $C$.

Given $\lambda\in F$, let $J_n(\lambda)$ stand for the lower
triangular $\lambda$-Jordan block, and write $J_n$ for $J_n(0)$.
Thus
$$
J_1(\lambda)=(\lambda),\quad J_2(\lambda)=\left(\begin{matrix} \lambda & 0\\
1 & \lambda\end{matrix}\right),\quad J_3(\lambda)=\left(\begin{matrix} \lambda & 0 & 0\\
1 & \lambda & 0\\
0 & 1 & \lambda
\end{matrix}\right),...
$$

\section{Gabriel's theorem}
\label{sec:existence}

The following is Waterhouse's \cite{WW} formulation of Gabriel's
theorem \cite{G}, except that we have replaced their blocks by
0-Jordan cells.

\begin{theorem} \label{matr} (i) Any matrix $A\in M_n(F)$ is congruent to the direct sum of
finitely many 0-Jordan blocks and an invertible matrix~$C$ (we
allow, of course, for the possibility that one, but not both, of
these types of summands be absent).

(ii) The multiplicity of each 0-Jordan block appearing in such
decomposition and the congruence type of $C$ are uniquely
determined by $A$.

(iii) For each integer $r\ge 1$ there is a unique matrix in
$M_r(F)$, up to congruence, which is non-invertible and
indecomposable, namely the 0-Jordan block $J_r$.
\end{theorem}

Observe that part {\it (iii)} follows from parts {\it (i)} and
{\it (ii)}. We refer the reader to \cite{DS} for a short and
conceptual proof of part {\it (ii)}.

\subsection{Existence of the decomposition} Here we give a new and very short proof
of part {\it (i)} of Theorem \ref{matr}.

By induction on $n$. We may assume that $A$ is non-invertible with
zero radical and $n>2$. In this case, we easily see that
$$
A\sim \left(%
\begin{array}{ccc}
  0 & 0 & 0 \\
  I & 0 & 0 \\
  0 & u & B \\
\end{array}%
\right),
$$
where $B\in M_m(F)$, $u$ is a column vector in $F^m$, and $m$ is
chosen as small as possible subject to $1\leq m\leq n-2$.

If $m=1$ then either $A\sim J_n$ or $A\sim J_{n-1}\oplus J_1$.
Assume $m>1$.

If $B$ had an invertible component we could split it off and we
would be done by induction. In the other case, $B\sim
C=J_{k_1}\oplus\cdots\oplus J_{k_s}$, and we see that
$$
A\sim \left(%
\begin{array}{ccc}
  0 & 0 & 0 \\
  I & 0 & 0 \\
  0 & w & C \\
\end{array}%
\right),
$$
where $w=w_1+\cdots+w_s$, $w_i\in F^{k_i}$, and all entries of
$w_i$ after the first are 0.

If $s=1$ then $A\sim J_n$ or $A\sim J_{n-m}\oplus J_m$. Suppose
$s>1$. If the first entry of any $w_i$, $2\leq i\leq s$, is zero
we can split off one 0-Jordan block and be done by induction.
Suppose, if possible, that the first entry of $w_2$ is not 0. It
follows that
$$
A\sim \left(%
\begin{array}{ccc}
  0 & 0 & 0 \\
  I & 0 & 0 \\
  0 & z & D \\
\end{array}%
\right),
$$
where now $D\in M_{m-1}(F)$ and $z$ is a column vector in
$F^{m-1}$, a contradiction.

\section{The asymmetry of a non-degenerate bilinear form}
\label{sec:riehm}

We assume in this section that $f$ is non-degenerate and follow
\cite{R} very closely. The asymmetry of $f$ is the only
$\si\in\GL(V)$ satisfying
\begin{equation}
\label{defas} f(v,u)=f(u,\si v),\quad u,v\in V.
\end{equation}

It is immediately verified that equivalent non-degenerate bilinear
forms have similar asymmetries. The rest of the paper is devoted
to see up to what extent the converse is also true. That this
fails, in general, is seen by the existence of alternating and
non-alternating symmetric bilinear forms in characteristic 2.

\begin{lemma}\label{simit} The asymmetry $\si$ is similar
to its inverse $\si^{-1}$.
\end{lemma}

\begin{proof} Let $\B$ be a basis of $V$, and let $A$ and $S$ be the
corresponding matrices of $f$ and~$\si$. Then
$$
A'=AS,\quad S=A^{-1}A'.
$$
Therefore
$$
S^{-1}=(A')^{-1}A=(A')^{-1}S'A'.
$$
Since a square matrix is always similar to its transpose, the
result follows.
\end{proof}

\begin{lemma}\label{desq} Suppose $U$ is a $\si$-invariant subspace of $V$ such that $f_U$
is non-degenerate. Then
$$
V=U\perp U^\perp.
$$
\end{lemma}

\begin{proof} We have
$$
f(U,v)=0\Leftrightarrow f(v,\si U)=0\Leftrightarrow f(v,U)=0,
$$
so
$$
L(U)=U^\perp=R(U).
$$
Since $f$ is non-degenerate, we infer
$$
\dim(U^\perp)=\dim(V)-\dim(U).
$$
From the fact that $f_U$ is non-degenerate, we deduce
$$
U\cap U^\perp=0.
$$
Combining the above we obtain the desired result.
\end{proof}

\begin{lemma}
\label{lem21a} If $v,w\in V$ then $f(\sigma v,\sigma w)=f(v,w)$.
\end{lemma}
\begin{proof} By (\ref{defas}), we have
$$
f(\si v,\si w)=f(w,\si v)=f(v,w).
$$
\end{proof}
It follows from Lemma \ref{lem21a} that
$$
f(\si v,w)=f(v,\si^{-1} w),\quad v,w\in V
$$
and, more generally,
$$ f(p(\si)v,w)=f(v,p(\si^{-1}) w),\quad q\in
F[X], v,w\in V.
$$

For $0\neq p\in F[X]$, let $p^* \in F[X]$ be its adjoint
polynomial, defined by
$$
p^*(X)=X^{\mathrm{deg} (p)}p(1/X).
$$
Thus, if $0\neq p\in F[X]$ has degree $k$, then
$$
 f(p(\si)v,w)=f(v,\si^{-k} p^*(\si) w),\quad v,w\in
V.
$$
This will be repeatedly and implicitly used throughout the paper.

We will view $V$ as an $F[X]$-module via $\si$. Let $p_\si\in
F[X]$ stand for the minimal polynomial of $\si$. As we only
concerned with split forms, we assume for the remainder of this
section that $p_\si$ splits over $F$. Note that by Lemma
\ref{simit}, if $\lambda\in F$ is an eigenvalue of $\si$, then so
is $\lambda^{-1}$.

For each $\lambda\in F$ let $V_\lambda$ denote the corresponding
generalized eigenspace of $\si$, i.e.,
$$
V_\lambda=\{v\in V\,|\, (X-\lambda)^m v=0\text{ for some }m\geq
1\}.
$$
Thus, if $\lambda_1,\dots,\lambda_t$ are the distinct eigenvalues
of $\si$ in $F$, then
\begin{equation}
\label{funda4} V=V_{\lambda_1}\oplus\cdots\oplus V_{\lambda_t}.
\end{equation}

In this regard, we have the following two results from \cite{R}.

\begin{lemma}
\label{lem83} If $\lambda,\mu\in F$ satisfy $\lambda\mu\neq 1$
then $V_\lambda$ and $V_\mu$ are orthogonal, i.e.,
$f(V_\lambda,V_\mu)=f(V_\mu,V_\lambda)=0$.
\end{lemma}

It follows from (\ref{funda4}) and Lemma \ref{lem83} that $V$ is
the orthogonal direct sum of non-degenerate bilinear spaces two
types: $V_\lambda$ and $V_\lambda\oplus V_{\lambda^{-1}}$, where
$\lambda^2=1$ in the first case, and $\lambda^2\neq 1$ in the
second.

\begin{lemma}\label{caca} Suppose $V=V_\lambda\oplus V_{\lambda^{-1}}$. Then the
equivalence type of $f$ is determined by the similarity type of
its asymmetry $\si$. In matrix terms, there is a basis of $V$
relative to which the Gram matrix of $f$ is
$$
A=\left(%
\begin{array}{cc}
  0 & B \\
  I & 0 \\
\end{array}%
\right),
$$
where $B$ is the matrix of $\si$ restricted to $V_\lambda$ or
$V_{\lambda^{-1}}$ with respect to some basis, these choices
yielding congruent matrices. In particular, we can choose $B$ to
be the direct sum of Jordan blocks with the same eigenvalue,
either $\lambda$ or $\lambda^{-1}$.
\end{lemma}

\begin{proof} Let $\B=\{v_1,\dots,v_m\}$ be a basis of $V_\lambda$
and let $B$ the matrix of the restriction of $\si$ to $V_\lambda$
relative to $\B$. By Lemma \ref{lem83} the linear maps
$V_\lambda\to  V_{\lambda^{-1}}^* $ and $V_{\lambda^{-1}}\to
V_{\lambda}^* $ induced by $f$ are isomorphisms. In particular,
there exists a basis $\C=\{w_1,\dots,w_m\}$ of $V_{\lambda^{-1}}$
such that $f(w_i,v_j)=\delta_{ij}$. Then
$\{v_1,\dots,v_m,w_1,\dots,w_m\}$ is a basis of $V$ and relative
to this basis the Gram matrix of $f$ is
$$
A=\left(%
\begin{array}{cc}
  0 & B' \\
  I & 0 \\
\end{array}%
\right).
$$
Since $B'$ is similar to $B$ and the roles of $V_\lambda$ and
$V_{\lambda^{-1}}$ are interchangeable, the result follows.
\end{proof}

\begin{lemma}\label{grama} Let $\lambda\in F$ be an eigenvalue of $\si$ such
that $\lambda^2\neq 1$. Then $f$ is indecomposable if and only if
the elementary divisors of $\si$ are
$(X-\lambda)^m,(X-\lambda^{-1})^m$, where $2m=n$. Moreover, in
that case $f$ admits both Gram matrices
$$
\left(%
\begin{array}{cc}
  0 & J_m(\lambda) \\
  I & 0 \\
\end{array}%
\right),\,\left(%
\begin{array}{cc}
  0 & J_m(\lambda^{-1}) \\
  I & 0 \\
\end{array}%
\right).
$$
\end{lemma}

\begin{proof} Suppose first $f$ is indecomposable.
Let $(X-\lambda)^m$ be the highest power of
$X-\lambda$ present in $p_\si$. Since $\si$ is similar to
$\si^{-1}$, we see that $(X-\lambda^{-1})^m$ is the highest power
of $X-\lambda^{-1}$ present in $p_\si$.

There exists a vector $v\in V$ having $\si$-minimal polynomial
$(X-\lambda)^m$. Since $(X-\lambda)^{m-1} v\neq 0$, Lemma
\ref{lem83} ensures the existence of $w\in V_{\lambda^{-1}}$ such
that
$$
f((X-\lambda)^{m-1}v,w)\neq 0.
$$
Clearly $w$ has $\si$-minimal polynomial $(X-\lambda^{-1})^m$. Let
$S=F[X]v$, $T=F[X]w$ and $U=S\oplus T$. We claim that $U=V$. By
Lemma \ref{desq}, since $f$ is indecomposable, it suffices to
verify that $f_U$ is non-degenerate. By Lemma \ref{lem83}, we have
$$
f(S,S)=0=f(T,T).
$$
Moreover, the induced linear map $S\to T^*$ is clearly injective,
and hence bijective. Let $t\in T$ and suppose $f(t,S)=0$. Then
$f(s,\si t)=0$ for all $s\in S$. Thus $\alpha(\si t)=0$ for all
$\alpha\in T^*$, so $\si t=0$, whence $t=0$. Thus $T\to S^*$ is
also a linear isomorphism. This shows that $f_U$ is non-degenerate
and proves the claim.

The converse is obvious from Lemma \ref{simit}, while the last
assertion follows from Lemma \ref{caca}.
\end{proof}

\section{The associated symmetric or alternating bilinear form}
\label{sec:oneortwo}

We assume in this section that $f$ is non-degenerate with
asymmetry $\si$, whose only eigenvalues are $\pm 1$. Thus the
minimal polynomial of $\si$ is $p^r$, $p=X\pm 1$, for some $r\geq
1$. According to \cite{R}, we have
$$
V=V_1\perp\cdots\perp V_r,
$$
where each $V_m$, $1\leq m\leq r$, is a free $F[X]/(p^m)$-module,
not uniquely determined by $\si$ unless $r=1$. Let
$$
V(m)=\{v\in V\,|\, p^m v=0\},\quad 0\leq m\leq r.
$$
It is easy to verify -see \cite{R} for details- that
$$
V_m/p V_m\cong V(m)/(V(m+1)+pV(m-1)),\quad 1\leq m\leq r,
$$
via the linear isomorphism
$$
v+pV_m\mapsto v+V(m+1)+pV(m-1).
$$
In particular, the restriction of $f$ to each $V_m$, $1\leq m\leq
r$, is a non-degenerate bilinear whose equivalence class is
uniquely determined by $f$.

The next result is also easy to verify and can be found in
\cite{R}.

\begin{lemma}\label{gorro} Suppose $V=V_m$. Then the bilinear form $\widehat{f}$
on $V/p V$ given by
$$
\widehat{f}(u+pV,v+pV)=f(p^{m-1}u,v),\quad u,v\in V
$$
is non-degenerate with scalar asymmetry given by $(-1)^{m-1}$ if
$p=X-1$, and $(-1)^{m}$ if $p=X+1$.
\end{lemma}

\begin{lemma}
\label{uno} Suppose $f$ is indecomposable. Then

(i) $V$ is a free $F[X]/(p^m)$ module for some $m\geq 1$.

(ii) $\si$ has one elementary divisor if and only $\widehat{f}$ is
non-alternating.

(iii) $\si$ has two elementary divisors if and only $\widehat{f}$
is alternating.
\end{lemma}

\begin{proof} The first assertion follows immediately from the above discussion.

Suppose $\widehat{f}$ is non-alternating. Then there is $v\in V$
such that
$$
f(p^{m-1}v,v)\neq 0.
$$
Let $U=F[X]v$. Clearly the minimal polynomial of $v$ is $p^m$. The
left radical of $f_U$ is an $F[X]$-submodule of $U$. Since the
minimum $F[X]$-submodule of $U$ is not in the left radical, the
latter must be 0. Thus $f_U$ is non-degenerate. By Lemma
\ref{desq}, this implies $V=U\perp U^\perp$. But $f$ is
indecomposable, so $V=U$, which implies that $\si$ has a single
elementary divisor.

Suppose next $\widehat{f}$ is alternating. Let $v_1,\dots,v_k$ be
a basis for $V$ over $F[X]/(p^m)$. It yields an $F$-basis of $V/p
V$. Since $\widehat{f}$ is non-degenerate and
$\widehat{f}(v_1+pV,v_1+pV)=0$, there is $1<i\leq k$ such that
$\widehat{f}(v_1+pV,v_i+pV)=0$. Let $S=F[X]v_1$ and $T=F[X]v_i$.
Then $S\cap T=0$. Since $f(p^{m-1}v_1,v_i)\neq 0$, the induced
$F$-linear map $S\to T^*$ is an isomorphism. Therefore so is $T\to
S^*$, as in the proof of Lemma \ref{grama}. Suppose $s+t$ is in
the left radical of $S\oplus T$. Then $f(s,T)=0$ and $f(t,S)=0$,
which implies $s=0=t$. It follows from Lemma \ref{desq} that
$V=S\oplus T$.
\end{proof}

\begin{cor} Suppose $f$ is indecomposable and $\chr(F)\neq
2$. Then $f$ has two elementary divisors if and only if $p=X-1$
and $m$ is even, or $p=X+1$ and $m$ is odd.
\end{cor}

\begin{cor} Suppose that $V=V_m$, that $\si$ has two elementary divisors, and that $\widehat{f}$ is
alternating.  Then $f$ is indecomposable.
\end{cor}

\begin{cor} Suppose that $\chr(F)=2$ and $f$ has Gram matrix
$$\left(%
\begin{array}{cc}
  0 & J_m(1) \\
  I_m & 0 \\
\end{array}%
\right)
$$
for some $m\geq 1$. Then $f$ is indecomposable.
\end{cor}

\begin{note} Suppose $\chr(F)\neq 2$ and $m$ is
odd in the above example. Then $\widehat{f}$ is symmetric and
hence non-alternating, so $f$ has 2 components, each with a single
elementary divisor $(X-1)^m$ by Lemma \ref{uno}.
\end{note}

\begin{lemma} Suppose $\chr(F)=2$, $V=V_m$ and $m$ is even. Then
$\widehat{f}$ is alternating.
\end{lemma}

\begin{proof} It suffices to show this when $f$ is indecomposable. Consider the bilinear form $g:V\times V\to F$ given by
$$g(u,v)=f(u,v)-f(v,u)=f(u,v)-f(u,\si v)=f(u,(1-\si)v),\quad u,v\in V.
$$
It is alternating, and hence of even rank. Since $f$ is non-degenerate, the rank of $g$ is the rank of $\si-1$, namely
$m-1$ times the number of elementary divisors of~$\si$. Since $m-1$ is odd, $\si$ must have an even number of elementary divisors.
This argument is due to Wall \cite{W}.

As $f$ is indecomposable, Lemma \ref{uno} implies that $\widehat{f}$ is alternating.
\end{proof}

\section{The non-alternating case}\label{sec:one}

\begin{theorem} Suppose that $f$ is non-degenerate and its asymmetry $\si$
has a single elementary divisor $p^n$, where $p=X\pm 1$ and $n\geq 1$. Then
either $p=X-1$ and $n$ is odd, or $p=X+1$, $\chr(F)\neq 2$ and $n$ is even.
Moreover, if $F$ is quadratically closed then the equivalence class of $f$
is uniquely determined by the similarity type of $\si$.
\end{theorem}

\begin{proof}\label{totu} The first assertion follows from \S\ref{sec:oneortwo}. Suppose
that $F$ is quadratically closed and that $f:V\times V\to F$ and $g:V\times V\to F$
are non-degenerate whose asymmetries $\si$ and $\tau$ are similar and have a single
elementary divisor $p^n$, where $p=X\pm 1$ and $n\geq 1$. We wish to show
that $f$ is equivalent to $g$. We argue by induction on $n$.

By a cyclic basis of $V$ (resp. $W$) we mean a basis of the form $u,pu,\dots,p^{n-1}u$,
where $u$ is a cyclic vector of $V$ (resp. $W$) with respect to $\si$ (resp. $\tau$).

Since $p^n$ annihilates $V$ (resp. $W$), it is immediately
verified that the Gram matrix of $f$ (resp. $g$) relative to a
cyclic basis is skew upper triangular, i.e., all entries below the
skew diagonal, which runs between positions $(1,n)$ and $(n,1)$,
are 0. Since $f$ (resp. $g$) is non-degenerate, all entries along
the skew diagonal are non-zero.

The base cases are $n=1$ when $p=X-1$, and $n=2$ when $p=X+1$, $\chr(F)\neq 2$. Since $F$ is quadratically closed, relative to suitable cyclic bases, $f$ and $g$ admit Gram matrices $(1)$ in the first case, and $\left(
                                                            \begin{array}{cc}
                                                              1 & 2 \\
                                                              -2 & 0 \\
                                                            \end{array}
                                                          \right)$ in the second.

Suppose $n\geq 3$ when $p=X-1$, and $n\geq 4$ when $p=X+1$, $\chr(F)\neq 2$. Suppose also the result is true in smaller dimensions.
Let $\widetilde{f}$ and $\widetilde{g}$ be the bilinear forms on $\widetilde{V}=V/p^{n-1}V$ and $\widetilde{W}=W/p^{n-1}W$ induced by $f$ and $g$.
By inductive hypothesis there is an isometry  $\Omega:\widetilde{V}\to \widetilde{W}$. Let $v\in V$ be a $\si$-cyclic vector. Then
$$
\Omega(pv+p^{n-1}V)=pw+p^{n-1}W
$$
for some $w\in W$. Now $\widetilde{V}$ and $\widetilde{W}$ are $F[X]$-modules via $\widetilde{\si}$ and $\widetilde{\tau}$, the maps induced by $\si$ and $\tau$,
so $\Omega$ is an $F[X]$-module isomorphism. Since $p^{n-3}(pv+p^{n-1}V)$ is not zero in $\widetilde{V}$, we see that $w$ is a $\tau$-cyclic
vector of $W$. Let $\B$ and $\C$ be the cyclic bases of $V$ and $W$ generated by $v$ and $w$. Since $\Omega$ is an isometry
and $p^n$ annihilates $V$ and $W$, we see that the Gram matrices of $f$ and $g$ relative to these bases become identical
once their first row and column is removed.

\noindent{\sc Case 1.} $p=X-1$ and $\chr(F)\neq 2$. Consider the linear system with $n-1$ equations in $n$ variables
$$
g(u,pw)=f(v,p v),\dots,g(u,p^{n-1}w)=f(v,p^{n-1}v).
$$
Since $pw,\dots,p^{n-1}w$ are linearly independent and $g$ is non-degenerate, the system has a solution, say $u_0\in W$.
Note that $p^{n-1}w$ is a solution to the associated homogeneous system, so $u=u_0+tp^{n-1}w$ is a solution to the original system
for any $t\in F$. Choose $t$ so that
$$
f(v,v)=g(u,u)=g(u_0,u_0)+2tg(u_0,p^{n-1}w).
$$
This is possible since $g(u_0,p^{n-1}w)=f(v,p^{n-1}v)$ is a skew diagonal entry and hence non-zero. Moreover, since
$g(u,p^{n-1}w)\neq 0$, it follows that $u\notin pW$, so $u,pw,\dots,p^{m-1}w$ is a basis of $W$. We easily verify that
the Gram matrix of $g$ relative to this basis is equal to the Gram matrix of $f$ relative to $v,pv,\dots,p^{n-1}v$.

\noindent{\sc Case 2.} $p=X+1$ and $n\geq 4$ (whatever $\chr(F)$ is). Consider the linear system with $n-2$ equations in $n$ variables
$$
g(u,p^2w)=f(v,p^2 v),\dots,g(u,p^{n-1}w)=f(v,p^{n-1}v).
$$
Since $p^2w,\dots,p^{n-1}w$ are linearly independent and $g$ is non-degenerate, the system has a solution, say $u_0\in W$.
Note that $p^{n-2}w$ is a solution to the associated homogeneous system, so $u=u_0+tp^{n-2}w$ is a solution to the original system
for any $t\in F$. Choose $t$ so that
$$
f(v,v)=g(u,u)=g(u_0,u_0)+tg(u_0,p^{n-1}w).
$$
Next let $z=pw+sp^{n-1}w$, where $s\in F$ is chosen so that
$$
f(v,pv)=f(u,z)=f(u,pw)+sf(u,p^{n-1}w).
$$
Then $u,z,p^2w,\dots,p^{n-1}w$ is a basis of $W$ and we easily verify that
the Gram matrix of $g$ relative to this basis is equal to the Gram matrix of $f$ relative to $v,pv,\dots,p^{n-1}v$.

\noindent{\sc Case 3.} $p=X+1$, $\chr(F)=2$ and $n=3$. Since $F$ is quadratically closed,
it is trivial to verify that the Gram matrix of $f$ relative to a suitable cyclic basis is
$$
\left(
\begin{array}{ccc}
a & 0 & 1 \\
1 & 1 & 0 \\
1 & 0 & 0 \\
\end{array}
\right).
$$
Choose $x\in F$ so that $x^2+x+a=0$. Then
$$
\left(
\begin{array}{ccc}
1 & x & 0 \\
0 & 1 & 0 \\
0 & 0 & 1 \\
\end{array}
\right)\left(
\begin{array}{ccc}
a & 0 & 1 \\
1 & 1 & 0 \\
1 & 0 & 0 \\
\end{array}
\right)\left(
\begin{array}{ccc}
1 & 0 & 0 \\
x & 1 & 0 \\
0 & 0 & 1 \\
\end{array}
\right)=\left(
\begin{array}{ccc}
0 & x & 1 \\
x+1 & 1 & 0 \\
1 & 0 & 0 \\
\end{array}
\right)\sim \left(
\begin{array}{ccc}
0 & 0 & 1 \\
1 & 1 & 0 \\
1 & 0 & 0 \\
\end{array}
\right).
$$
This shows that $f$ is determined by $\si$.
\end{proof}

The existence of at least one bilinear form $f$ whose asymmetry
has a single elementary divisor $p^n$, where $p=X-1$ and $n$ is
odd, or $p=X+1$, $\chr(F)\neq 2$ and $n$ is even, was proven by
Wall \cite{W}. The proof of Theorem \ref{totu} actually gives a
recursive mechanism to produce \emph{ all } Gram matrices of such
$f$, with respect to a fixed cyclic basis, i.e., when $\si$ has
matrix $J_n(1)$ if $p=X=1$ and $n$ is odd, and matrix $J_n(-1)$ if
$p=X+1$, $\chr(F)\neq 2$ and $n$ is even. Of course, all such
forms are equivalent, but our point here is that the proof of
Theorem \ref{totu} can be used to individually describe all such
$f$ with respect to a fixed cyclic basis. We omit the details for
two reasons. They are not relevant to our goals and, most
importantly, not a single one of these Gram matrices has a
particularly pleasant shape. Unfortunately, that is also the case
if we use, instead, a basis $v,\si v,\dots,\si^{n-1}v$, in which
case the matrix of $\si$ is the companion matrix to $p^n$ and
every $f$ with asymmetry $\si$ has a Toeplitz Gram matrix. By far,
the simplest known type of Gram matrix is $\Gamma_n$ when
$\chr(F)\neq 2$ and $\Gamma_n^0$ when $\chr(F)=2$, as given in
\cite{HS3}. It is immediately verified that if $\chr(F)\neq 2$,
then the asymmetry $\Gamma_n^{-1}\Gamma_n'$ of $\Gamma_n$ has a
single elementary divisor, namely $(X-1)^n$ if $n$ is odd and
$(X+1)^n$ if $n$ is even, while if $\chr(F)=2$ and $n$ is odd then
the asymmetry of $\Gamma_n^0$ has a single elementary divisor
$(X-1)^n$. In all of these cases the asymmetry has an unremarkable
shape, which makes the actual finding of $\Gamma_n$ and
$\Gamma_n^0$ intriguing.

\section{The alternating case}
\label{sec:two}

\begin{theorem}\label{tota} Suppose $f$ is non-degenerate and indecomposable, and
that its asymmetry $\si$ has 2 elementary divisors $p^m,p^m$,
where $p=X\pm 1$ and $m\geq 1$.

(i) If $\chr(F)\neq 2$ then $f$ has Gram matrix
$$
\left(%
\begin{array}{cc}
  0 & J_m(\mp 1)' \\
  I_m & 0 \\
\end{array}%
\right)
$$
relative to some basis of $V$, so the equivalence class of $f$ is
completely determined by the similarity type of $\si$.

(ii) The same conclusion follows if $\chr(F)=2$ and $F$ has no
separable quadratic extensions.
\end{theorem}

\begin{proof} By induction on $m$. If $m=1$ then, by Lemma
\ref{uno}, $f$ has Gram matrix
$$
\left(%
\begin{array}{cc}
  0 & \mp 1 \\
  1 & 0 \\
\end{array}%
\right)
$$
when $p=X\pm 1$. Suppose $m>1$ and the result is true for
exponents less than $m$. Let $\widetilde{V}=pV/p^{m-1} V$. Then
the restriction of $f$ to $pV$ induces a non-degenerate bilinear
form  $\widetilde{f}$ on  $\widetilde{V}$ with asymmetry
$\widetilde{\si}$, the map that $\si$ induces on  $\widetilde{V}$.
The elementary divisors of $\si$ on  $\widetilde{V}$ are
$p^{m-2},p^{m-2}$. Lemma \ref{uno} ensures that  $\widetilde{f}$
is indecomposable.

By inductive hypothesis there are $v,w\in V$ such that
$\widetilde{v},\widetilde{w}$ form a basis of the
$F[X]/(p^{m-2})$-module  $\widetilde{V}$ and the
$F[X]/(p^{m-2})$-modules they generate are totally isotropic for
$\widetilde{f}$. We easily see that $v,w$ form a basis of the
$F[X]/(p^{m})$-module $V$.

Let $S=F[X]v$, which has $F$-basis $v,pv,\dots,p^{m-1}v$, and let
$T=F[X]w$. Since $\widehat{f}$, as defined in Lemma \ref{gorro},
is alternating, the map $S\to T^*$ induced by $f$ must be a linear
isomorphism. Let $w_0,w_1,\dots,w_{m-1}$ be an $F$-basis of $T$
dual to $v,pv,\dots,p^{m-1}v$. Then
$$
f(p^{i+1}w_i,p^j v)= f(w_i,\si^{-(i+1)}(p^*)^{i+1}p^j v)=0,\quad
0\leq i,j\leq m-1,
$$
so $p^{i+1}w_i=0$, i.e., $w_i\in p^{m-(i+1)}T$. It follows that
$f$ annihilates the subspaces generated by $pv,\dots,p^{m-1}v$ and
$w_0,\dots,w_{m-2}$. Moreover, if $p=X-1$ and $u$ is in $S$ or
$T$, we have
$$
f(u,(\si-1)u)=f(u,\si u)-f(u,u)=0,
$$
which readily implies
$$
f(u,(\si-1)^i u)=0,\quad i\geq 1.
$$
This, in turn, yields
$$
f((\si-1)u,u)=-f(u,\si^{-1}(\si-1)u)=0
$$
and therefore
$$
f((\si-1)^iu,u)=0,\quad i\geq 1.
$$
On the other hand, if $p=X+1$ and $u$ is in $S$ or $T$, we have
$$
f(u,(\si+1)^i u)=0=f((\si+1)^i u, u),\quad i\geq 2,
$$
and
$$
f(v,(\si+1)v)=2f(v,v), f((\si+1)v,v)=-2f(v,v).
$$
All in all, relative to the $F$-basis
$v,pv,\dots,p^{m-1}v,w_0,w_1,\dots,w_{m-1}$ of $V$, the Gram
matrix of $f$ is equal to
$$
C=\left(%
\begin{array}{cc}
  B & J' \\
  I & A \\
\end{array}%
\right),\quad J=J_m(\mp 1).
$$
Here all entries of $B$ (resp. $A$), except possibly for $B_{11}$
(resp. $A_{nn}$), are 0 when $p=X-1$, while all entries of $B$
(resp. $A$), except possibly for $B_{11},B_{12},B_{21}$ (resp.
$A_{nn},A_{n-1,n},A_{n,n-1}$), are 0 if $p=X+1$, in which case
$B_{12}=2B_{11}$ and $B_{21}=-2B_{11}$. In order to relate
$A_{n-1,n}$ and $A_{n,n-1}$ to $A_{n,n}$ when $p=X+1$, note that
in this case the matrix of $\si$ relative to the above basis of
$V$ is
$$
S=\left(%
\begin{array}{cc}
  J & 0 \\
  0 & U \\
\end{array}%
\right),\quad J=J_m(-1),
$$
where
$$
CS=C'.
$$
This forces $U=(J')^{-1}$ and $A(J')^{-1}=A'$, which gives
$A_{n-1,n}=-2A_{n,n}$  and $A_{n,n-1}=2A_{n,n}$.

The rest of the proof essentially reduces to the cases: $m=2$ and
$p=X-1$; $m=3$, $p=X+1$ and $\chr(F)\neq 2$. Note that when $m=2$
and $p=X-1$ we arrive at the same conclusions as above directly
through Lemma \ref{uno} and the use of dual bases, without
resorting to an inductive argument.

Suppose first $m=2$ and $p=X-1$. Then
$$
C=\left(%
\begin{array}{cccc}
  a & 0 & 1 & 1 \\
  0 & 0 & 0 & 1 \\
  1 & 0 & 0 & 0 \\
  0 & 1 & 0 & b \\
\end{array}%
\right).
$$
If $\chr(F)\neq 2$ we can use two obvious elementary congruence
transformations (E.C.T.) to eliminate $a$ and $b$.

Suppose $\chr(F)=2$ and $F$ satisfies the stated hypothesis. If
exactly one of $a,b$ is 0, we can easily eliminate the other by
means of two E.C.T. Suppose both are different from 0. For $x\in
F$, add $x$ times column 4 to column 1 of $C$ and then add $x$
times row 4 to row 1 of the resulting matrix. We find $bx^2+x+a$
in entry (1,1). By hypothesis we may choose $x\in F$ so that
$bx^2+x+a=0$. This will eliminate entry (1,1). The rest of the
argument is routine.

Suppose next $m=3$, $p=X+1$ and $\chr(F)\neq 2$. Then
$$
\left(%
\begin{array}{cccccc}
  a & 2a & 0 & -1 & 1 & 0 \\
  -2a & 0 & 0 & 0 & -1 & 1 \\
  0 & 0 & 0 & 0 & 0 & -1 \\
  1 & 0 & 0 & 0 & 0 & 0 \\
  0 & 1 & 0 & 0 & 0 & 2\alpha \\
  0 & 0 & 1 & 0 & -2\alpha & \alpha \\
\end{array}%
\right).
$$
Subtract $a$ times column 5 from column 1 and then do the same to
row 1. This eliminates entry (1,1) of the resulting matrix.
Another E.C.T. eliminates entries (1,2) and (2,1). The rest of the
argument is routine.
\end{proof}

\begin{note} Our hypothesis on $F$ (placed only when $\chr(F)=2$)
means that $F$ has no separable or Galois extensions of degree
$2^t$ for any $t\geq 1$. This certainly holds if $F$ is
quadratically closed. Note, however, that the separable closure of
$F_2(X)$ has no Galois extensions of degree $>1$ but is not
quadratically closed.
\end{note}

\begin{note} If $\chr(F)=2$ and $F$ has a separable quadratic extension then the
conclusion of Theorem \ref{tota} fails. Indeed, suppose
$q=X^2+\alpha X+\beta\in F[X]$ is irreducible, where $\alpha\neq
0$. Let $b=1/\alpha$ and $a=\beta/\alpha$. Consider the quadratic
$Q$ form over $F$:
$$
Q(X_1,X_2)=aX_1^2+bX_2^2+X_1X_2.
$$
Since $q$ is irreducible, $Q$ does not represent 0. In particular,
$Q$ is not equivalent to $X_1X_2$. Now $Q$ is the quadratic form
on $V/p V$ associated to the bilinear form $f$ with Gram matrix
$C$, where $a,b$ are chosen as above, via
$$
Q(u+pV,u+pV)=f(pu,u), \quad u\in V,
$$
while $X_1X_2$ corresponds to the choice $a=0=b$. Thus, these two
choices for $C$ render non-equivalent bilinear forms.
\end{note}

\section{Classification and normal forms of split forms}
\label{sec:main}

Our results from \S\ref{sec:oneortwo}, \S\ref{sec:one} and
\S\ref{sec:two} readily yield a classification of all split forms
as well as a list of split normal forms, when $F$ is quadratically
closed. The statements below are in perfect agreement with the
corresponding ones from \cite{R} and \cite{HS3} when $F$ is
algebraically closed. Of course, in this case, the results stated
below are applicable to all bilinear forms. In view of Theorem
\ref{matr}, we may restrict to the case of non-degenerate forms.
Recall the definitions of $\Gamma_n$ and $\Gamma_n^0$, as given in
\cite{HS3}.
$$
\Gamma_1=(1),\,\Gamma_2=\left(\begin{array}{cc}
0&-1\\
1&1
\end{array}\right),\,\Gamma_3= \left(\begin{array}{ccc}
0&0&1\\
0&-1&-1\\
1&1&0
\end{array}\right),\,\Gamma_4= \left(\begin{array}{cccc}
0&0&0&-1\\
0&0&1&1\\
0&-1&-1&0\\
1&1&0&0
\end{array}\right),\dots
$$
$$
\Gamma_1^0=(1),\,\Gamma_3^0= \left(\begin{array}{ccc}
0&0&1\\
0&1&0\\
1&1&0
\end{array}\right),\,
\Gamma_5^0= \left(\begin{array}{ccccc}
0&0&0&0&1\\
0&0&0&1&0\\
0&0&1&0&0\\
0&1&1&0&0\\
1&1&0&0&0
\end{array}\right),\dots
$$

\begin{theorem}\label{t1} Suppose that $F$ is quadratically closed.
Let $V$ be a $F$-vector space of finite dimension $n\geq 1$. Then
the equivalence class of a non-degenerate, indecomposable and
split form $f:V\times V\to F$ is uniquely determined by the
similarity type of its asymmetry $\si$. Moreover, if $f$ admits
any of the following as Gram matrix, then $f$ is a non-degenerate,
indecomposable and split form and $\si$ has the indicated
elementary divisors:
$$
A_n=\Gamma_n, n\text{ odd}, \chr(F)\neq 2, (X-1)^n,
$$
$$
B_n=\Gamma_n^0, n\text{ odd}, \chr(F)=2, (X-1)^n,
$$
$$
C_n=\Gamma_n, n\text{ even}, \chr(F)\neq 2, (X+1)^n,
$$
$$
D_n=\left(
               \begin{array}{cc}
                 0 & J_m(1) \\
                 I_m & 0 \\
               \end{array}
             \right), n=2m, m\text{ even}, (X-1)^m,
             (X-1)^m,
$$
$$
E_n=\left(
               \begin{array}{cc}
                 0 & J_m(1) \\
                 I_m & 0 \\
               \end{array}
             \right), n=2m, m\text{ odd}, \chr(F)=2, (X-1)^m,
             (X-1)^m,
$$
$$
F_n=\left(
               \begin{array}{cc}
                 0 & J_m(-1) \\
                 I_m & 0 \\
               \end{array}
             \right),n=2m, m \text{ odd}, \chr(F)\neq 2, (X+1)^m,
             (X+1)^m,
$$
$$
G_n(\lambda)=\left(
               \begin{array}{cc}
                 0 & J_m(\lambda) \\
                 I_m & 0 \\
               \end{array}
             \right)
, 0\neq \lambda\in F, n=2m, (X-\lambda)^m,(X-\lambda^{-1})^m.
$$
Furthermore, any non-degenerate, indecomposable and split form on
$V$ admits one and only of these as Gram matrix, except only for
the fact that $G_n(\lambda)$ and $G_n(\lambda^{-1})$ represent
equivalent forms.
\end{theorem}

\begin{theorem}\label{mami} Suppose that $F$ is quadratically closed. Let $f,g:V\times V\to F$ be
non-degenerate split forms with respective asymmetries $\si$ and
$\tau$. Suppose $\si$ and $\tau$ are similar. Then

(i) If $\chr(F)\neq 2$, or $\chr(F)=2$ and no $(X-1)^m$ with $m$
odd is an elementary divisor of $\si$, then $f$ and $g$ are
equivalent.

(ii) Suppose that $\chr(F)=2$ and $(X-1)^m$, for some odd $m$, is
an elementary divisor of $\si$. Let $V(\si)$ (resp. $V(\tau))$ be
the generalized 1-eigenspace of $V$ with respect to $\si$ (resp.
$\tau$), and consider any $f$-orthogonal (resp. $g$--orthogonal)
decomposition
$$
V(\si)=U_1\perp\cdots\perp U_r\, (\text{resp. }
V(\si)=W_1\perp\cdots\perp W_r),
$$
where each $U_m$ (resp. $W_m$) is a free $F[X]/((X-1)^m)$-module
via $\si$ (resp. $\tau$), $1\leq m\leq r$. For each odd $m$ such
that $1\leq m\leq r$, let $f_m$ (resp. $g_m$) be the
non-degenerate symmetric bilinear form on $U_m/(\si-1)U_m$ (resp.
$W_m/(\tau-1)W_m$)
$$
f_m(u+(\si-1)U_m,v+(\si-1)U_m)=f((\si-1)^{m-1}u,v)\, \quad u,v\in
U_m
$$
$$
( \text{resp. }
g_m(u+(\tau-1)W_m,v+(\tau-1)W_m)=g((\tau-1)^{m-1}u,v), \quad
u,v\in W_m).
$$
Then $f$ and $g$ are equivalent if and only if for each odd $m$,
$1\leq m\leq r$, such that $U_m\neq 0$, $f_m$ and $g_m$ are both
alternating or both non-alternating.
\end{theorem}

\begin{note} Suppose $F$ is algebraically closed. Then Theorem \ref{mami} makes it obvious that every
square matrix over $F$ is congruent to its transpose, i.e., that
$f$ is equivalent to its transpose $f'$, defined by
$f'(u,v)=f(v,u)$. Indeed, since $J_m\sim J_m'$ (trivial),
Gabriel's decomposition reduces the result to the case when $f$ is
non-degenerate. Let $\si$ be the asymmetry of $f$. Then $f'$ has
asymmetry $\si^{-1}$. Since $\si$ is similar to $\si^{-1}$, it
follows at once from Theorem \ref{mami} (whether $\chr(F)=2$ or
not) that $f$ is equivalent to $f'$.
\end{note}

\begin{theorem} Suppose that $F$ is quadratically closed and that $f$ is
a non-degenerate split form. Then $f$ admits as Gram matrix the
direct sum of indecomposable matrices taken from Theorem \ref{t1},
without simultaneously using summands $B_{\ell}$ and $E_{2\ell}$,
$\ell$ odd, when $\chr(F)=2$. The summands of such decomposition
are uniquely determined by $f$, except only for the fact that
$G_m(\lambda)$ and $G_m(\lambda^{-1})$ are interchangeable.
\end{theorem}

\noindent{\bf Acknowledgements.} I thank V. Sergeichuk for useful
comments. I dedicate this paper to D. Djokovic in appreciation for
the valuable times spent together discussing all sorts of
algebraic problems. I learnt the basic ideas of this particular
subject from him.

\noindent{\bf Added in Proof.} Complex matrix representatives
under congruence were already known to Turnbull and Aitken in
1932. See [31, p. 139] for details. I am grateful to F. de Ter\'an
for this information.

\end{document}